\documentclass{amsart}

\usepackage[section]{placeins}

\usepackage{amsmath}             
\usepackage{amscd}
\usepackage{amssymb}             
\usepackage{amsfonts}            
\usepackage{latexsym}            
\usepackage{amsthm}              
\usepackage{graphicx}

\usepackage{enumerate}
\usepackage{mathrsfs} 
\usepackage{thmtools}
\usepackage{thm-restate}

\newcommand{\Sp}{\operatorname{Sp}}
\newcommand{\SL}{\operatorname{SL}}
\newcommand{\GL}{\operatorname{GL}}
\newcommand{\SO}{\operatorname{SO}}

\newcommand{\D}{\mathrm{d}}
\newcommand{\g}{\mathfrak{g}}

\newtheorem{lause}{Theorem}[section]

\newtheorem{lemma}[lause]{Lemma}
\newtheorem{seur}[lause]{Corollary}
\newtheorem{prop}[lause]{Proposition}
\newtheorem{prob}[lause]{Problem}
\newtheorem*{lause*}{Theorem}

\theoremstyle{definition}

\theoremstyle{remark}
\newtheorem{remark}[lause]{Remark}
\newtheorem{esim}[lause]{Example} 
\newtheorem*{mot*}{Motivation}
\newtheorem*{acknow*}{Acknowledgements}

\numberwithin{equation}{section}

\usepackage{hyphenat}
\allowdisplaybreaks

\begin{document}

\title[Jordan blocks of nilpotent elements]{Jordan blocks of nilpotent elements in some irreducible representations of classical groups in good characteristic}

\author{Mikko Korhonen}
\address{Department of Mathematics, Southern University of Science and Technology, \text{Shenzhen} 518055, Guangdong, P. R. China}
\email{korhonen\_mikko@hotmail.com}
\thanks{Partially supported by SNSF grant P2ELP2\_181902, and NSFC grants 11771200 and 11931005.}


\date{\today}

\begin{abstract}

Let $G$ be a classical group with natural module $V$ and Lie algebra $\mathfrak{g}$ over an algebraically closed field $K$ of good characteristic. For rational irreducible representations $f: G \rightarrow \operatorname{GL}(W)$ occurring as composition factors of $V \otimes V^*$, $\wedge^2(V)$, and $S^2(V)$, we describe the Jordan normal form of $\mathrm{d} f(e)$ for all nilpotent elements $e \in \mathfrak{g}$. The description is given in terms of the Jordan block sizes of the action of $e$ on $V \otimes V^*$, $\wedge^2(V)$, and $S^2(V)$, for which recursive formulae are known. Our results are in analogue to earlier work (Proc. Amer. Math. Soc., 147 (2019) 4205--4219), where we considered these same representations and described the Jordan normal form of $f(u)$ for every unipotent element $u \in G$. 


\end{abstract}

\maketitle

\section{Introduction}

Let $G$ be a simple algebraic group over an algebraically closed field $K$ of characteristic $p > 0$, with Lie algebra $\g$. Recall that an element $e \in \g$ is \emph{nilpotent}, if $\D f(e)$ is a nilpotent linear map for every rational representation $f: G \rightarrow \GL(W)$. In this paper, we consider some special cases of the following problem.

\begin{prob}\label{prob:mainnilpotent}
Let $f: G \rightarrow \GL(W)$ be a rational irreducible representation and let $e \in \g$ be a nilpotent element. What is the Jordan normal form of $\D f(e)$?
\end{prob}

Let $G$ be a simple classical group ($\SL(V)$, $\Sp(V)$, or $\SO(V)$) and assume that $p$ is \emph{good for $G$}. In other words, we assume $p > 2$ if $G = \Sp(V)$ or $G = \SO(V)$. Let $f: G \rightarrow \GL(W)$ be a rational irreducible representation with highest weight $\lambda$. As the main result of this paper, we describe the Jordan normal form of $\D f(e)$ for every nilpotent element $e \in \g$ in the following cases:

\begin{itemize}
\item $G = \SL(V)$ and $\lambda = \varpi_1 + \varpi_{n-1}$, where $n = \dim V$. (Theorem \ref{thm:nilmainthmA})
\item $G = \Sp(V)$ and $\lambda = \varpi_2$. (Corollary \ref{corollary:nilSPmain})
\item $G = \SO(V)$ and $\lambda = 2 \varpi_1$. (Corollary \ref{corollary:nilSOmain})
\end{itemize}

These irreducible representations are found as composition factors of the tensor product $V \otimes V^*$, the exterior square $\wedge^2(V)$, and the symmetric square $S^2(V)$, respectively. Our results and proofs are analogous to our previous work in \cite{KorhonenJordanGood}, where we considered the Jordan normal form of $f(u)$ for unipotent elements $u \in G$. 

As one application of the results in this paper, we get a description of the Jordan normal form of $\operatorname{ad}(e)$ for every nilpotent element $e \in \g$ when $G$ is a simple algebraic group of type $A_n$ (Remark \ref{remark:adjointmaps}). Combining this results from the literature, the Jordan normal form of $\operatorname{ad} (e)$ is known whenever $G$ is a simple algebraic group in good characteristic (Remark \ref{remark:adjointmaps2} and \cite{StewartNilpotentBlocks}).

Before stating our main results, we establish some basic notation that will be used throughout the paper. Let $T$ be an indeterminate. For $\lambda \in K$ and an integer $d > 0$, denote by $J_d(\lambda)$ an indecomposable $K[T]$-module on which the $T$ acts as a single $d \times d$ Jordan block with eigenvalue $\lambda$. We set $J_0(\lambda) = 0$ for all $\lambda \in K$. For a $K$-vector space $V$ we denote $V^0 = 0$ and $V^n = V \oplus \cdots \oplus V$ ($n$ copies) for $n > 0$. Then if we have $e$ acting nilpotently on a $K$-vector space $V$, notation such as $V \downarrow K[e] = J_{d_1}(0)^{n_1} \oplus \cdots \oplus J_{d_t}(0)^{n_t}$ can be used to say that $e$ acts on $V$ with Jordan blocks of sizes $d_1$, $\ldots$, $d_t$, with block size $d_i$ occurring with multiplicity $n_i$. For a linear map $f: V \rightarrow V$ which is either unipotent or nilpotent, we denote by $r_m(f)$ the number of Jordan blocks of size $m$ in the Jordan decomposition of $f$. We say that $f$ is \emph{regular}, if its Jordan normal form consists of a single Jordan block. We denote by $\nu_p$ the $p$-adic valuation on the integers, so $\nu_p(a)$ is the largest integer $k \geq 0$ such that $p^k$ divides $a$.

Fix a maximal torus $S$ of $G$ with character group $X(S)$, and a base $\Delta$ for the root system of $G$. We denote the dominant weights corresponding to $\Delta$ by $X(S)^+$. The $i$th fundamental dominant weight is denoted by $\varpi_i$, using the standard Bourbaki labeling of the simple roots \cite[11.4, p. 58]{Humphreys}. We denote the rational irreducible $G$-module with highest weight $\lambda \in X(S)^+$ by $L_G(\lambda)$. 

The following theorem is our main result for $G = \SL(V)$, and it will be proven in Section \ref{section:proofsection}. It describes the Jordan normal form of each nilpotent element $e \in \g$ on $L_G(\varpi_1 + \varpi_{n-1})$ in terms of the Jordan normal form of $e$ on $V \otimes V^*$. There is no closed formula for the Jordan block sizes of $e$ on $V \otimes V^*$, but various recursive formulae exist in the literature, which we discuss in Section \ref{section:lietensorblock}. We state the theorem by identifying as $G$-modules $V \otimes V^* \cong \mathfrak{gl}(V)$ and $L_G(\varpi_1 + \varpi_{n-1}) \cong \mathfrak{psl}(V) = \mathfrak{sl}(V) / Z(\mathfrak{sl}(V))$, where $n = \dim V$.

\begin{lause}\label{thm:nilmainthmA}
Let $G = \SL(V)$, where $\dim V = n$ for some $n \geq 2$. Let $e \in \g$ be a nilpotent element and $V \downarrow K[e] = J_{d_1}(0) \oplus \cdots \oplus J_{d_t}(0)$, where $t \geq 1$ and $d_r \geq 1$ for all $1 \leq r \leq t$. Set $\alpha = \nu_p(\gcd(d_1, \ldots, d_t))$. Let $e_0$ be the action of $e$ on $\mathfrak{gl}(V)$, let $e_0'$ be the action of $e$ on $\mathfrak{sl}(V)$, and let $e_0''$ be the action of $e$ on $\mathfrak{psl}(V)$. 

The Jordan block sizes of $e_0'$ are determined from those of $e_0$ as follows:

\begin{enumerate}[\normalfont (i)]
\item If $\alpha = 0$, then $r_1(e_0') = r_1(e_0) - 1$ and $r_m(e_0') = r_m(e_0)$ for all $m \neq 1$.
\item If $\alpha > 0$, then $r_{p^{\alpha}-1}(e_0') = 1$, $r_{p^{\alpha}}(e_0') = r_{p^{\alpha}}(e_0)-1$, and $r_m(e_0') = r_m(e_0)$ for all $m \neq p^{\alpha}, p^{\alpha} - 1$.
\end{enumerate}

The Jordan block sizes of $e_0''$ are determined from those of $e_0$ as follows:

\begin{enumerate}[\normalfont (i)]
\setcounter{enumi}{2}
\item If $p \nmid n$, then $r_1(e_0'') = r_1(e_0) - 1$ and $r_m(e_0'') = r_m(e_0)$ for all $m \neq 1$.
\item If $p \mid n$ and $\alpha = 0$, then $r_1(e_0'') = r_1(e_0) - 2$ and $r_m(e_0'') = r_m(e_0)$ for all $m \neq 1$.
\item If $p \mid n$ and $\alpha > 0$, then $r_{p^{\alpha}}(e_0'') = r_{p^{\alpha}}(e_0) - 2$, $r_{p^{\alpha}-1}(e_0'') = 2$ and $r_m(e_0'') = r_m(e_0)$ for all $m \neq p^{\alpha}, p^{\alpha} - 1$.
\end{enumerate}
\end{lause}

\begin{remark}\label{remark:adjointmaps}As a corollary of Theorem \ref{thm:nilmainthmA}, we have a description of the Jordan normal form of $\operatorname{ad}(e) : \g \rightarrow \g$ for every nilpotent element $e \in \g$ when $G$ is a simple algebraic group of type $A_{n-1}$. To see this, let $\pi: G_{sc} \rightarrow G$ be an isogeny, where $G_{sc} = \SL(V)$ with $\dim V = n$ is the simply connected cover of $G$. Let $\sigma: G \rightarrow G_{ad}$ be an isogeny into the group $G_{ad} = \operatorname{PGL}(V)$ of adjoint type $A_{n-1}$. 

The structure of $\g$ has been described by Hogeweij, and it follows from the results in \cite[Table 1, p. 446]{Hogeweij} that as $G$-modules we have an isomorphism $$\g \cong \begin{cases}\mathfrak{sl}(V), & \text { if $\pi$ is separable.}  \\ \mathfrak{pgl}(V), & \text { if $\sigma$ is separable.} \\ \mathfrak{psl}(V) \oplus K, & \text { if both $\pi$ and $\sigma$ are inseparable.}  \end{cases}$$ Note that $\g \cong \mathfrak{psl}(V) \oplus K$ can only occur if $p^2 \mid n$. 

Let $e_0 \in \g_{sc}$ be nilpotent. As $G_{sc}$-modules $\mathfrak{gl}(V) \cong V \otimes V^*$. There are various recursive formulae for calculating the Jordan normal form of $e_0$ on $V \otimes V^*$ which we discuss in Section \ref{section:lietensorblock}, so we assume that the Jordan block sizes of $e_0$ on $\mathfrak{gl}(V)$ are known. Then the Jordan block sizes of the action of $e_0$ on $\mathfrak{sl}(V)$ are given by Theorem \ref{thm:nilmainthmA} (i) -- (ii), and the block sizes are the same on $\mathfrak{pgl}(V)$ since $\mathfrak{pgl}(V) \cong \mathfrak{sl}(V)^*$ as $G$-modules. Furthermore, the Jordan block sizes of $e_0$ on $\mathfrak{psl}(V) \oplus K$ are given by Theorem \ref{thm:nilmainthmA} (iii) -- (v). Since for any nilpotent element $e \in \g$ there exists a unique nilpotent element $e_0 \in \g_{sc}$ such that $\D \pi(e_0) = e$, we have a description of the Jordan normal form of $\operatorname{ad}(e)$ in all cases.

With similar arguments, using our previous work in \cite{KorhonenJordanGood}, one can describe the Jordan normal form of $\operatorname{Ad}(u) : \g \rightarrow \g$ for every unipotent element $u \in G$ when $G$ is simple of type $A_{n-1}$. Note that although \cite[Theorem 6.1]{KorhonenJordanGood} does not state the Jordan normal form of a unipotent element $u_0 \in G_{sc}$ on $\mathfrak{sl}(V)$, this is calculated in the proof \cite[p. 4215]{KorhonenJordanGood} --- see Corollary \ref{thm:nilmainthmCOR} and its proof in Section \ref{section:proofsection}.\end{remark}

Our main results for $G = \Sp(V)$ and $G = \SO(V)$ are given in the following corollaries of Theorem \ref{thm:nilmainthmA}, which are analogous to \cite[Corollary 6.2 -- 6.3]{KorhonenJordanGood}. The proofs will be given in Section \ref{section:proofsection}. The results are given in terms of Jordan block sizes of nilpotent elements $e \in \g$ in their action on $\wedge^2(V)$ and $S^2(V)$, we discuss the formulae for these in Section \ref{section:lietensorblock}.

\begin{seur}\label{corollary:nilSPmain}Assume $p > 2$, and let $G = \Sp(V)$, where $\dim V = n$ for some $n \geq 4$. Let $e \in \g$ be a nilpotent element and $V \downarrow K[e] = J_{d_1}(0) \oplus \cdots \oplus J_{d_t}(0)$, where $t \geq 1$ and $d_r \geq 1$ for all $1 \leq r \leq t$. Set $\alpha = \nu_p(\gcd(d_1, \ldots, d_t))$. Let $e_0$ be the action of $e$ on $\wedge^2(V)$, and let $e_0''$ be the action of $e$ on $L_G(\varpi_2)$. Then the Jordan block sizes of $e_0''$ are determined from those of $e_0$ by the rules (iii) -- (v) of Theorem \ref{thm:nilmainthmA}.\end{seur}

\begin{seur}\label{corollary:nilSOmain}Assume $p > 2$, and let $G = \SO(V)$, where $\dim V = n$ for some $n \geq 5$. Let $e \in \g$ be a nilpotent element and $V \downarrow K[e] = J_{d_1}(0) \oplus \cdots \oplus J_{d_t}(0)$, where $t \geq 1$ and $d_r \geq 1$ for all $1 \leq r \leq t$. Set $\alpha = \nu_p(\gcd(d_1, \ldots, d_t))$. Let $e_0$ be the action of $e$ on $S^2(V)$, and let $e_0''$ be the action of $e$ on $L_G(2\varpi_1)$. Then the Jordan block sizes of $e_0''$ are determined from those of $e_0$ by the rules (iii) -- (v) of Theorem \ref{thm:nilmainthmA}.\end{seur}

We finish this introduction with a few comments about the connection between results for Problem \ref{prob:mainnilpotent} and for the analogous problem for unipotent elements \cite[Problem 1.1]{KorhonenJordanGood}. 

In the case where $G$ is simply connected and $p$ is good for $G$, it was shown by Springer \cite[Theorem 3.1]{SpringerIsomorphism} that there exists a $G$-equivariant isomorphism $\varepsilon: \mathcal{U}(G) \rightarrow \mathcal{N}(\g)$ between the unipotent variety of $G$ and the nilpotent variety of $\g$, called a \emph{Springer isomorphism}. As a corollary of Springer's result, an elementary argument shows that more generally $\varepsilon$ exists if $p$ is good and the isogeny $G_{sc} \rightarrow G$ from the simply connected cover of $G$ is a separable morphism \cite[Section 7]{SobajeExponential}. A Springer isomorphism is not unique, but the bijection it induces between the unipotent conjugacy classes of $G$ and the nilpotent orbits in $\g$ is unique by a result of Serre \cite[Appendix]{McNinchOptimal}. 



For example for $G = \SL(V)$ we identify $\g = \mathfrak{sl}(V)$, and in this case one example of a Springer isomorphism is given by $\varepsilon(u) = u-1$ for all $u \in \mathcal{U}(G)$. For other classical groups in characteristic $p > 2$, an example of a Springer isomorphism is given by the Cayley transform $u \mapsto (1-u)(1+u)^{-1}$. It follows that if $G$ is a classical group in good characteristic with natural module $V$, then $u$ and $\varepsilon(u)$ have the same Jordan block sizes on $V$ for all $u \in \mathcal{U}(G)$. This is true more generally in characteristic zero, since in this case it follows from theorems of Jacobson--Morozov and Kostant that for any irreducible $G$-module $W$, the Jordan block sizes of $u$ and $\varepsilon(u)$ on $W$ are the same for all $u \in \mathcal{U}(G)$. 

In positive characteristic this is certainly false in general, since it is possible for example that $\D f = 0$. It is more interesting to look at the case where $\D f$ is an irreducible representation of $\g$, which is precisely the case where the highest weight of $f$ is \emph{$p$-restricted} \cite[Theorem 7.5 (iii)]{BorelLectures}. For $p$-restricted irreducible representations one can also find examples where the Jordan block sizes of $f(u)$ and $\D f(\varepsilon(u))$ differ \cite[Corollary 1.2]{StewartNilpotentBlocks}. As a positive result, when $f$ is the adjoint representation in good characteristic and $G$ is not of type $A_{kp-1}$, it is known that $f(u)$ and $\D f(\varepsilon(u))$ have the same Jordan block sizes \cite[Theorem 4.1]{PremetStewartSheets}.

Thus for a rational irreducible $p$-restricted representation $f: G \rightarrow \GL(W)$, it makes sense to compare the Jordan block sizes of $f(u)$ and $\D f(\varepsilon(u))$ for $u \in \mathcal{U}(G)$ and to examine to what extent they are similar. We proceed to do this for the irreducible representations considered in Theorem \ref{thm:nilmainthmA} and \cite[Theorem 6.1]{KorhonenJordanGood}. It turns out in these cases that the Jordan block sizes of $f(u)$ and $\D f(\varepsilon(u))$ are very similar, but not always exactly the same.


For example, consider $G = \SL(V)$ with $n = \dim V$ and suppose that $p \mid n$. Let $\alpha = \nu_p(n)$, and consider a regular unipotent element $u \in G$ and a regular nilpotent element $e \in \g$. By Proposition \ref{prop:tensorproductcompare} below, the Jordan block sizes of $u$ and $e$ are the same on $\mathfrak{gl}(V) \cong V \otimes V^*$. Furthermore, by \cite[Lemma 4.2]{KorhonenJordanGood} the smallest Jordan block size of $u$ on $V \otimes V^*$ is $p^{\alpha}$, occurring with multiplicity $p^{\alpha}$. Thus by Theorem \ref{thm:nilmainthmA} and \cite[Theorem 6.1]{KorhonenJordanGood}, we have \begin{align*}L_G(\varpi_1 + \varpi_{n-1}) \downarrow K[u] &= J_{p^{\alpha}-2}(1) \oplus J_{p^{\alpha}}(1)^{p^{\alpha}-1} \oplus \bigoplus_{d > p^{\alpha}} J_d(1)^{r_d}, \\ L_G(\varpi_1 + \varpi_{n-1}) \downarrow K[e] &= J_{p^{\alpha}-1}(0)^2 \oplus J_{p^{\alpha}}(0)^{p^{\alpha}-2} \oplus \bigoplus_{d > p^{\alpha}} J_d(0)^{r_d}\end{align*} for some integers $r_d \geq 0$.

More generally, we have the following result, which will be proven in Section \ref{section:proofsection}.

\begin{seur}\label{thm:nilmainthmCOR}
Let $G = \SL(V)$, where $\dim V = n$ for some $n \geq 2$. Let $e \in \g$ be a nilpotent element and $V \downarrow K[e] = J_{d_1}(0) \oplus \cdots \oplus J_{d_t}(0)$, where $t \geq 1$ and $d_r \geq 1$ for all $1 \leq r \leq t$. Set $\alpha = \nu_p(\gcd(d_1, \ldots, d_t))$. Let $u \in G$ be a unipotent element with the same Jordan block sizes on $V$ as $e$, so $V \downarrow K[u] = J_{d_1}(1) \oplus \cdots \oplus J_{d_t}(1)$. Then

\begin{enumerate}[\normalfont (i)]
\item $u$ and $e$ have the same Jordan block sizes in their action on $\mathfrak{gl}(V)$ and $\mathfrak{sl}(V)$,
\item $u$ and $e$ have the same Jordan block sizes on $L_G(\varpi_1 + \varpi_{n-1}) \cong \mathfrak{psl}(V)$ if and only if $p^{\alpha+1} \mid n$.
\end{enumerate}

\end{seur}




\section{Preliminaries}

In this section, we list some preliminary results, all of which are well known.

\begin{lemma}\label{lemma:easybinomial}
Let $p$ be a prime number and $\beta \geq 0$. Then $\binom{p^{\beta}-1}{t} \equiv (-1)^t \mod{p}$ for all $0 \leq t \leq p^{\beta}-1$.
\end{lemma}

\begin{proof}Similarly to the proof in \cite[Lemma 3.1]{KorhonenJordanGood}, one can proceed by induction on $t$, using the fact that $\binom{p^\beta}{t} \equiv 0 \mod{p}$ for all $0 < t < p^\beta$.\end{proof}

\begin{lemma}[{\cite[Lemma 3.3]{KorhonenJordanGood}}]\label{jordanrestrictionNIL}
Let $e \in \mathfrak{gl}(V)$ be a nilpotent linear map. Suppose that $W \subseteq V$ is a subspace invariant under $e$ such that $\dim V/W = 1$. Let $m \geq 0$ be such that $\operatorname{Ker} e^m \subseteq W$ and $\operatorname{Ker} e^{m+1} \not\subseteq W$. Then 

\begin{enumerate}[\normalfont (a)]

\item if $m = 0$, we have

\begin{itemize}
\item $r_1(e_W) = r_1(e) - 1$,
\item $r_i(e_W) = r_i(e)$ for all $i \neq 1$.
\end{itemize}

\item if $m \geq 1$, we have

\begin{itemize}
\item $r_{m+1}(e_W) = r_{m+1}(e) - 1$,
\item $r_m(e_W) = r_m(e) + 1$,
\item $r_i(e_W) = r_i(e)$ for all $i \neq m, m+1$.
\end{itemize}

\end{enumerate}

\end{lemma}


%

\begin{lemma}[{\cite[1.14]{SeitzClassical}, \cite[Proposition 4.6.10]{McNinch}}]\label{lemma:typeAomega}
Let $G = \SL(V)$, where $\dim V = n$ for some $n \geq 2$. If $p \nmid n$, then $V \otimes V^* \cong L_G(\varpi_1 + \varpi_{n-1}) \oplus L_G(0)$ as $G$-modules. If $p \mid n$, then $V \otimes V^*$ is uniserial with a filtration $V \otimes V^* \supset Z \supset Z'$ such that $V \otimes V^* / Z \cong L_G(0) \cong Z'$ and $Z/Z' \cong L_G(\varpi_1 + \varpi_{n-1})$.\end{lemma}

\begin{lemma}[{\cite[Lemma 3.7]{KorhonenJordanGood}}]\label{lemma:typeAomegabasisV}
Let $G = \SL(V)$ and set $n = \dim V$. Let $v_1, \ldots, v_n$ be a basis of $V$ and let $v_1^*$, $\ldots$, $v_n^*$ be the corresponding dual basis of $V^*$. Then $\sum_{1 \leq i \leq n} v_i \otimes v_i^*$ spans the unique $1$-dimensional $G$-submodule of $V \otimes V^*$.
\end{lemma}

\section{Jordan block sizes in tensor products}\label{section:lietensorblock}

Let $G = \SL(V)$, $G = \Sp(V)$, or $G = \SO(V)$, and let $\rho: G \rightarrow \GL(V)$ be the natural representation. In the main results of this paper, we describe the Jordan normal form of a nilpotent element $e \in \g$ in certain irreducible representations of $G$. The results are given in terms of the Jordan normal form of $\D (\rho \otimes \rho^*)(e)$, $(\D \wedge^2(\rho))(e)$ (for $p > 2$), and $(\D S^2(\rho))(e)$ (for $p > 2$). The purpose of this section is to describe the how Jordan normal forms of these linear maps can be calculated using results from the literature, so none of the results in this section are new.

Note that $e$ has the same Jordan normal form on $V$ and $V^*$, hence the Jordan block sizes of $e$ on $V \otimes V^*$ and $V \otimes V$ are the same. As a first step of the calculation, write $V = W_1 \oplus \cdots \oplus W_t$, where each $W_i$ is $e$-invariant and $W_i \downarrow K[e] \cong J_{d_i}(0)$ for some $d_i > 0$. As $K[e]$-modules, we have \begin{align} \label{eq:vxv} V \otimes V &\cong \bigoplus_{1 \leq i,j \leq t} W_i \otimes W_j, \\ \label{eq:altv} \wedge^2(V) &\cong \bigoplus_{1 \leq i \leq t} \wedge^2(W_i) \oplus \bigoplus_{1 \leq i < j \leq t} W_i \otimes W_j, \\ \label{eq:symv} S^2(V) &\cong \bigoplus_{1 \leq i \leq t} S^2(W_i) \oplus \bigoplus_{1 \leq i < j \leq t} W_i \otimes W_j. \end{align} This reduces the calculation to determining the Jordan normal form of $e$ on $W_i \otimes W_j$, and when $p > 2$, on $\wedge^2(W_i)$ and $S^2(W_i)$. 

Let $u \in G$ be a unipotent element such that each $W_i$ is $u$-invariant and $W_i \downarrow K[u] = J_{d_i}(1)$ for all $1 \leq i \leq t$. By the following proposition, the Jordan block sizes of $u$ and $e$ are the same on $W_i \otimes W_j$ for all $1 \leq i,j \leq t$.

\begin{prop}[{\cite[Section III]{Fossum}, \cite[Corollary 5 (a)]{NormanTwoRelated}}]\label{prop:tensorproductcompare}Let $V_1$ and $V_2$ be finite-dimensional vector spaces over $K$. For $i = 1,2$ let $u_i \in \SL(V_i)$ be regular unipotent and let $e_i \in \mathfrak{sl}(V_i)$ be regular nilpotent. Then $u_1 \otimes u_2$ and $e_1\otimes \operatorname{id}_{V_2} + \operatorname{id}_{V_1} \otimes e_2$ have the same Jordan block sizes on $V_1 \otimes V_2$.
\end{prop}

The Jordan normal form of the tensor product of unipotent matrices has been extensively studied in the literature \cite{Srinivasan}, \cite{Ralley}, \cite{McFall}, \cite{Renaud}, \cite{Norman}, \cite{Norman2}, \cite{Hou}, \cite{Barry}. There is no easy closed formula in positive characteristic, but there are various recursive formulae involving only calculations with the integers $d_i$ --- see for example \cite[Theorem 1]{Barry}. By Proposition \ref{prop:tensorproductcompare}, such formulae can also be used to calculate the Jordan normal form of $e$ on $W_i \otimes W_j$. With this, we have a complete method for calculating the Jordan normal form of $\D (\rho \otimes \rho^*)(e)$.


There are also recursive formulae for the Jordan normal form of $u$ on $\wedge^2(W_i)$ and $S^2(W_i)$, see \cite[Theorem 2]{Barry} for $p > 2$. By the next proposition, for $p > 2$ the Jordan block sizes of $u$ and $e$ are the same on $\wedge^2(W_i)$ and $S^2(W_i)$, so \cite[Theorem 2]{Barry} can be used to describe the Jordan normal form of $e$ on $\wedge^2(W_i)$ and $S^2(W_i)$. Hence we are able to calculate the Jordan normal form of $(\D \wedge^2(\rho))(e)$ and $(\D S^2(\rho))(e)$ when $p > 2$.

\begin{prop}[{\cite[Theorem 21]{McNinchAdjoint}}]\label{prop:odduninilregularwedge}
Assume that $p > 2$. Let $u \in \SL(V)$ be a regular unipotent element, and let $e \in \mathfrak{sl}(V)$ be a regular nilpotent element. Then $u$ and $e$ have the same Jordan block sizes on $\wedge^2(V)$ and $S^2(V)$.
\end{prop}

As an alternative to the more general setting of formal group laws taken in \cite{McNinchAdjoint}, Proposition \ref{prop:odduninilregularwedge} can also be proven by adapting the proof of Proposition \ref{prop:tensorproductcompare} given in \cite[Corollary 5 (a)]{NormanTwoRelated}.

\begin{remark}In characteristic $p = 2$, Proposition \ref{prop:odduninilregularwedge} fails. For example, let $\dim V = 4$ and consider a regular unipotent element $u \in \SL(V)$ and a regular nilpotent element $e \in \mathfrak{sl}(V)$. A computation shows that $\wedge^2(V) \downarrow K[u] = J_2(1) \oplus J_4(1)$ but $\wedge^2(V) \downarrow K[e] = J_3(0)^2$, and similarly $S^2(V) \downarrow K[u] = J_2(1) \oplus J_4(1)^2$ but $S^2(V) \downarrow K[e] = J_1(0)^2 \oplus J_4(0)^2$.\end{remark}

Summarizing the results from the literature described in this section, we get the following proposition from~\eqref{eq:vxv} -- \eqref{eq:symv} and Propositions \ref{prop:tensorproductcompare} -- \ref{prop:odduninilregularwedge}.


\begin{prop}[{\cite[Theorem 24]{McNinchAdjoint}}]\label{prop:nilunipcompare}
Let $u \in \SL(V)$ be a unipotent element, and let $e \in \mathfrak{sl}(V)$ be nilpotent with the same Jordan block sizes on $V$ as $u$. Then:

\begin{enumerate}[\normalfont (i)]
\item $u$ and $e$ have the same Jordan block sizes on $V \otimes V^*$.
\item If $p > 2$, then $u$ and $e$ have the same Jordan block sizes on $\wedge^2(V)$ and $S^2(V)$.
\end{enumerate} 
\end{prop}

\begin{remark}\label{remark:adjointmaps2}(cf. Remark \ref{remark:adjointmaps}) Suppose that $G$ is of classical type $B_n$, $C_n$, or $D_n$, and assume that $p > 2$. Using the methods described in this section, one can calculate the Jordan normal form of $\operatorname{ad}(e)$ and $\operatorname{Ad}(u)$ for all nilpotent $e \in \g$ and unipotent $u \in G$. 

Consider for example the case where $G$ is of type $B_n$ or $D_n$. Let $G_{sc} = \operatorname{Spin}(V)$ be the simply connected cover of $G$. In this case $\g$ is simple, and as $G_{sc}$-modules $\g \cong \wedge^2(V)$  \cite[1.14, 8.1 (a) -- (b)]{SeitzClassical}. Thus by using Proposition \ref{prop:nilunipcompare} (ii) and the results discussed in this section, we can describe the Jordan normal form of $\operatorname{ad}(e)$ and $\operatorname{Ad}(u)$ for all nilpotent $e \in \g$ and unipotent $u \in G$. In the case where $G$ is of type $C_n$, we can argue similarly, by considering $G_{sc} = \Sp(V)$ and using the fact that $\g \cong S^2(V)$ as $G_{sc}$-modules \cite[1.14, 8.1 (c)]{SeitzClassical}.\end{remark}


\section{Proof of main results}\label{section:proofsection}


In this section, we will prove our main results: Theorem \ref{thm:nilmainthmA} and Corollaries \ref{corollary:nilSPmain} - \ref{thm:nilmainthmCOR}. At the end we will also give a table of examples illustrating Theorem \ref{thm:nilmainthmA}.

We begin with some calculations concerning the action of a regular nilpotent element of $\SL(V)$ on $V \otimes V^*$. Fix a basis $v_1$, $\ldots$, $v_n$ of $V$, and let $v_1^*$, $\ldots$, $v_n^*$ be the corresponding dual basis of $V^*$, so $v_i^*(v_j) = \delta_{ij}$ for all $1 \leq i,j \leq n$. For convenience, we set $v_i = 0$ and $v_i^* = 0$ for all $i \leq 0$ and $i > n$. In the following lemmas we consider a regular nilpotent element $e \in \mathfrak{sl}(V)$ defined by \begin{equation}\label{eq:regularnilpaction}ev_i = v_{i-1}\end{equation} for all $1 \leq i \leq n$.

\begin{lemma}\label{lemma:nilpbinpower}
Let $k \geq 0$. Then $$e^k \cdot (v_i \otimes v_j^*) = \sum_{0 \leq t \leq k} \binom{k}{t} (-1)^{k+t} v_{i-t} \otimes v_{j+k-t}^*$$ for all $1 \leq i,j \leq n$.
\end{lemma}

\begin{proof}The action of $e$ on $V \otimes V^*$ is given by $x \otimes \operatorname{id}_{V^*} + \operatorname{id}_V \otimes y$, where $x$ and $y$ are the actions of $e$ on $V$ and $V^*$, respectively. Now the lemma follows by computing the action of $e^k$ on $V \otimes V^*$ with the binomial theorem.\end{proof}

\begin{lemma}\label{lemma:betapowerzero}
Suppose that $p^{\beta} \mid n$ for some $\beta \geq 0$ and write $n = p^{\beta} k_{\beta}$. Set $$\delta_{\beta}' = \sum_{1 \leq j \leq k_\beta} v_{jp^{\beta}} \otimes v_{jp^{\beta}}^*.$$ Then $e^{p^{\beta}} \cdot \delta_{\beta}' = 0$.
\end{lemma}

\begin{proof}Let $W$ be the subspace of $V$ spanned by $v_{jp^{\beta}}$ for all $1 \leq j \leq k_\beta$. It is clear that $W$ is invariant under the action of $e^{p^{\beta}}$. On the other hand, by Lemma \ref{lemma:typeAomegabasisV} the vector $\delta_{\beta}'$ is annihilated by the action of $\mathfrak{sl}(W)$, so we conclude that $e^{p^{\beta}} \cdot \delta_{\beta}' = 0$.\end{proof}

\begin{lemma}\label{lemma:betapowerm1}
Suppose that $p^{\beta} \mid n$ for some $\beta \geq 0$ and write $n = p^{\beta} k_{\beta}$. Set $$\delta_{\beta} = \sum_{0 \leq j \leq k_{\beta} - 1} v_{(j+1)p^{\beta}} \otimes v_{jp^{\beta}+1}^*.$$ Then $e^{p^{\beta}-1} \cdot \delta_\beta = \sum_{1 \leq t \leq n} v_t \otimes v_t^*$.
\end{lemma}

\begin{proof}
By Lemma \ref{lemma:nilpbinpower} and Lemma \ref{lemma:easybinomial}, we have 
\begin{align*} 
e^{p^{\beta}-1} \cdot (v_{(j+1)p^{\beta}} \otimes v_{jp^{\beta}+1}^*) &= \sum_{0 \leq t \leq p^{\beta}-1} \binom{p^{\beta}-1}{t} (-1)^{p^{\beta}-1+t} v_{(j+1)p^{\beta} - t} \otimes v_{(j+1)p^{\beta} - t}^* \\ 
&= \sum_{jp^{\beta}+1 \leq t \leq (j+1)p^{\beta}} v_t \otimes v_t^* \end{align*} 
for all $0 \leq j \leq k_{\beta}-1$, from which the lemma follows.
\end{proof}


\begin{proof}[Proof of Theorem \ref{thm:nilmainthmA}]We use a strategy similar to the proof of \cite[Theorem 6.1]{KorhonenJordanGood}. First we note that if $p \nmid n$, then $\mathfrak{gl}(V) = \mathfrak{sl}(V) \oplus Z(\mathfrak{gl}(V))$ and in particular $\mathfrak{psl}(V) \cong \mathfrak{sl}(V)$ as $G$-modules. In this case $e_0'$ and $e_0''$ have the same Jordan block sizes, and statements (i) -- (v) of the theorem are clear.

Suppose then that $p \mid n$. Write $V = W_{1} \oplus \cdots \oplus W_{t}$, where each $W_{i}$ is $e$-invariant and $W_{i} \downarrow K[e] = J_{d_i}(0)$. For $1 \leq r \leq t$, let $(v_j^{(r)})_{1 \leq j \leq d_r}$ be a basis of $W_{r}$, such that $ev_j^{(r)} = v_{j-1}^{(r)}$ for all $1 \leq j \leq d_r$, where we define $v_{0}^{(r)} = 0$. Let $(v_j^{(r)^*})$ be the dual basis of $V^*$ corresponding to the basis $(v_j^{(r)})$ of $V$. For $1 \leq r \leq t$, we define $$\gamma_r = \sum_{1 \leq j \leq d_r} v_j^{(r)} \otimes  v_j^{(r)^*}.$$ By Lemma \ref{lemma:typeAomegabasisV}, the element $\gamma = \sum_{1 \leq r \leq t} \gamma_r$ spans the unique $1$-dimensional $G$-submodule of $V \otimes V^*$. 

Throughout the proof we will identify $\mathfrak{gl}(V) \cong V \otimes V^*$. Under this isomorphism, the $G$-submodule corresponding to $\mathfrak{sl}(V)$ is $\operatorname{Ker} \varphi$, where $\varphi: V \otimes V^* \rightarrow K$ is defined by $\varphi(v \otimes f) = f(v)$ for all $v \in V$ and $f \in V^*$. Furthermore, the $G$-module corresponding to $\mathfrak{psl}(V)$ is $\operatorname{Ker} \varphi / \langle \gamma \rangle$, here $\gamma \in \operatorname{Ker} \varphi$ since $p \mid n$. Then $e_0$, $e_0'$, and $e_0''$ are the actions of $e$ on $V \otimes V^*$, $\operatorname{Ker} \varphi$, and $\operatorname{Ker} \varphi / \langle \gamma \rangle$, respectively.

Next we will apply Lemma \ref{jordanrestrictionNIL} to describe the Jordan block sizes of $e_0'$ and $e_0''$. By Proposition \ref{prop:nilunipcompare} (i) and \cite[Lemma 4.3]{KorhonenJordanGood}, the smallest Jordan block size of $e_0$ is $p^{\alpha}$. Thus \begin{equation}\label{eq:nilpminker1}\operatorname{Ker} e_0^{p^{\alpha}-1} \subset \operatorname{Ker} \varphi\end{equation} by Lemma \ref{jordanrestrictionNIL}. We will show that \begin{equation}\label{eq:nilpminker2}\operatorname{Ker} e_0^{p^{\alpha}} \not\subset \operatorname{Ker} \varphi,\end{equation} which together with~\eqref{eq:nilpminker1} and Lemma \ref{jordanrestrictionNIL} proves that the Jordan block sizes of $e_0'$ are as described in (i) -- (ii) of the theorem.

Let $1 \leq r' \leq t$ be such that $\nu_p(d_{r'}) = \alpha$, and write $d_{r'} = p^{\alpha}k_{r'}$. Consider $$v = \sum_{1 \leq j \leq k_{r'}} v_{jp^{\alpha}}^{(r')} \otimes v_{jp^{\alpha}}^{(r')^*}.$$ We have $e^{p^{\alpha}} \cdot v = 0$ by Lemma \ref{lemma:betapowerzero}, and $v \not\in \operatorname{Ker} \varphi$ since $\varphi(v) = k_{r'}$. This proves~\eqref{eq:nilpminker2}.


Next we consider the Jordan block sizes of $e_0''$. Let $f$ be the action of $e$ on $(V \otimes V^*) / \langle \gamma \rangle$. By Lemma \ref{lemma:typeAomega} $(\operatorname{Ker} \varphi)^* \cong (V \otimes V^*) / \langle \gamma \rangle$ as $G$-modules, so the Jordan block sizes of $f$ are the same as those of $e_0'$. Thus to prove that the Jordan block sizes of $e_0''$ are as described in (iv) -- (v) of the theorem, by Lemma \ref{jordanrestrictionNIL} it will suffice to show that \begin{align}\label{eq:fineqpalpham1} \operatorname{Ker} f^{p^{\alpha}-1} &\subset \operatorname{Ker} \varphi / \langle \gamma \rangle, \\ \label{eq:fineqpalpha} \operatorname{Ker} f^{p^{\alpha}} &\not\subset \operatorname{Ker} \varphi / \langle \gamma \rangle \end{align} hold. First we note that~\eqref{eq:fineqpalpha} is immediate from~\eqref{eq:nilpminker2}. For~\eqref{eq:fineqpalpham1}, for all $1 \leq r \leq t$ write $d_r = p^{\alpha} k_r$ and let $$\delta_r = \sum_{0 \leq j \leq k_r - 1} v_{(j+1)p^{\alpha}}^{(r)} \otimes v_{jp^{\alpha}+1}^{(r)^*}.$$ Since the action of $e$ on $W_r$ is as defined in~\eqref{eq:regularnilpaction}, it follows from Lemma \ref{lemma:betapowerm1} that $e^{p^{\alpha}-1} \cdot \delta_r = \gamma_r$. Hence for $\delta = \sum_{1 \leq r \leq t} \delta_r$, we have $e^{p^{\alpha}-1} \cdot \delta = \gamma$, so $\operatorname{Ker} f^{p^{\alpha}-1} = \left( \operatorname{Ker} e_0^{p^{\alpha}-1} + \langle \delta \rangle \right) / \langle \gamma \rangle.$ Furthermore $\delta \in \operatorname{Ker} \varphi$, so we conclude from~\eqref{eq:nilpminker1} that~\eqref{eq:fineqpalpham1} holds, which completes the proof of the theorem.\end{proof}

\begin{proof}[Proof of Corollary \ref{corollary:nilSPmain}]Similarly to the proof of \cite[Corollary 6.2]{KorhonenJordanGood}, with Theorem \ref{thm:nilmainthmA} the corollary follows from the isomorphisms $V \otimes V^* \cong \wedge^2(V) \oplus S^2(V)$ and $L_{\SL(V)}(\varpi_1 + \varpi_{n-1}) \cong L_G(\varpi_2) \oplus S^2(V)$ of $G$-modules.\end{proof}

\begin{proof}[Proof of Corollary \ref{corollary:nilSOmain}]As in \cite[Corollary 6.3]{KorhonenJordanGood}, the corollary follows using Theorem \ref{thm:nilmainthmA} and the fact that we have isomorphisms $V \otimes V^* \cong \wedge^2(V) \oplus S^2(V)$ and $L_{\SL(V)}(\varpi_1 + \varpi_{n-1}) \cong \wedge^2(V) \oplus L_G(2 \varpi_1)$ of $G$-modules.\end{proof}

\begin{proof}[Proof of Corollary \ref{thm:nilmainthmCOR}]Since $V \otimes V^* \cong \mathfrak{gl}(V)$, the claim in (i) about $\mathfrak{gl}(V)$ follows from Proposition \ref{prop:nilunipcompare} (i). The Jordan block sizes of the action of $u_0$ on $\mathfrak{sl}(V)$ are not explicitly stated in \cite[Theorem 6.1]{KorhonenJordanGood}, but are described in its proof. Indeed, on \cite[p. 4215]{KorhonenJordanGood} the action of $u$ on $\mathfrak{gl}(V)$ corresponds to $u_0$, while the action of $u$ on $\mathfrak{pgl}(V) = \mathfrak{gl}(V) / Z(\mathfrak{gl}(V))$ corresponds to $u_0'$. The arguments on \cite[p. 4215]{KorhonenJordanGood} show that the values of $r_m(u_0')$ are given exactly by the rules (i) -- (ii) of Theorem \ref{thm:nilmainthmA}. Since $\mathfrak{pgl}(V) \cong \mathfrak{sl}(V)^*$ as $G$-modules, we conclude that $u$ and $e$ have the same Jordan block sizes in their action on $\mathfrak{sl}(V)$.

Since $L_G(\varpi_1 + \varpi_{n-1}) \cong \mathfrak{psl}(V)$ as $G$-modules, claim (ii) follows easily by comparing Theorem \ref{thm:nilmainthmA} (iii) -- (v) and \cite[Theorem 6.1 (i) -- (iii)]{KorhonenJordanGood}.\end{proof}

\begin{esim}\label{esim:tablexample}Let $G = \SL(V)$ with $\dim V = n$ for $n \geq 2$. In Table \ref{table:examplesNIL} below, we give for all $2 \leq n \leq 5$ and all nilpotent elements $e \in \g = \mathfrak{sl}(V)$ the Jordan normal form of the action of $e$ on $V \otimes V^*$ and $L_G(\varpi_1 + \varpi_{n-1})$, in the case where $p \mid n$. In the table, we use notation $d_1^{n_1}, \ldots, d_t^{n_t}$ to denote $J_{d_1}(0)^{n_1} \oplus \cdots \oplus J_{d_t}(0)^{n_t}$, where $0 < d_1 < \cdots < d_t$ and $n_i \geq 1$ for all $1 \leq i \leq t$.\end{esim}

\begin{table}[!htbp]
\centering
\footnotesize
\caption{Theorem \ref{thm:nilmainthmA} in the cases $2 \leq n \leq 5$, see Example \ref{esim:tablexample}.}\label{table:examplesNIL}

\begin{tabular}{| c | l | l | l |}
\hline
$G$              & $V \downarrow K[e]$       & $V \otimes V^* \downarrow K[e]$         & $L_G(\varpi_1 + \varpi_{n-1}) \downarrow K[e]$ \\ \hline
$n = 2$, $p = 2$ & $2$ & $2^2$ & $1^2$ \\
                 & $1^2$ & $1^4$ & $1^2$ \\
& & & \\
$n = 3$, $p = 3$ & $3$        & $3^{3}$               & $2^{2}, 3$ \\
                 & $1, 2$     & $1^{2}, 2^{2}, 3$     & $2^{2}, 3$ \\
                 & $1^{3}$    & $1^{9}$               & $1^{7}$ \\
								 & & & \\
$n = 4$, $p = 2$ & $4$        & $4^{4}$               & $3^{2}, 4^{2}$ \\
                 & $1, 3$     & $1^{2}, 3^{2}, 4^{2}$ & $3^{2}, 4^{2}$ \\
                 & $2^{2}$    & $2^{8}$               & $1^{2}, 2^{6}$ \\
                 & $1^{2}, 2$ & $1^{4}, 2^{6}$        & $1^{2}, 2^{6}$ \\
                 & $1^{4}$    & $1^{16}$              & $1^{14}$ \\
								 & & & \\
$n = 5$, $p = 5$ & $5$        & $5^{5}$ & $4^2, 5^{3}$ \\
                 & $1, 4$     & $1^{2}, 4^{2}, 5^{3}$ & $4^{2}, 5^{3}$ \\
                 & $2, 3$     & $1^{2}, 2^{2}, 3^{2}, 4^{2}, 5$ & $2^{2}, 3^{2}, 4^{2}, 5$ \\
                 & $1^{2}, 3$ & $1^{5}, 3^{5}, 5$ & $1^{3}, 3^{5}, 5$ \\
                 & $1, 2^{2}$ & $1^{5}, 2^{4}, 3^{4}$ & $1^{3}, 2^{4}, 3^{4}$ \\
                 & $1^{3}, 2$ & $1^{10}, 2^{6}, 3$ & $1^{8}, 2^{6}, 3$ \\
                 & $1^{5}$    & $1^{25}$ & $1^{23}$ \\
\hline
\end{tabular}

\end{table}


\end{document}